\documentclass{amsart}

\usepackage{amssymb}
\usepackage{amsthm}
\usepackage{amsmath}
\usepackage{mathtools}
\usepackage{amssymb}
\usepackage{latexsym}
\usepackage{graphics}
\usepackage{graphicx}
\usepackage{color}
\usepackage{overpic}

\usepackage{verbatim}
\usepackage[shortlabels]{enumitem}

\usepackage[hidelinks,pdftex]{hyperref}

\AtBeginDocument{%
   \def\MR#1{}
}

\usepackage{clipboard}
\usepackage{cleveref}

\theoremstyle{plain}
\newtheorem{theorem}{Theorem}[section]
\newtheorem{lemma}[theorem]{Lemma}
\newtheorem{proposition}[theorem]{Proposition}
\newtheorem{corollary}[theorem]{Corollary}

\newtheorem*{namedtheorem}{\theoremname}
\newcommand{\theoremname}{testing}
\newenvironment{named}[1]{\renewcommand{\theoremname}{#1}\begin{namedtheorem}}{\end{namedtheorem}}

\theoremstyle{definition}
\newtheorem{definition}[theorem]{Definition}

\newtheorem{example}[theorem]{Example}
\newtheorem{construction}[theorem]{Construction}

\numberwithin{equation}{section}

\newcommand{\RR}{{\mathbb{R}}}

\newcommand{\from}{\colon} 

\newcommand{\bdy}{\partial}
\renewcommand{\setminus}{{\smallsetminus}}

\newcommand{\cut}{{\backslash\backslash}}
\newcommand{\Ddiagram}{\widetilde{\pi}(\widetilde{L})}

\newcommand{\refthm}[1]{Theorem~\ref{Thm:#1}}
\newcommand{\reflem}[1]{Lemma~\ref{Lem:#1}}

\newcommand{\refsec}[1]{Section~\ref{Sec:#1}}
\newcommand{\reffig}[1]{Figure~\ref{Fig:#1}}

\begin{document}

\title[Alternating links on nonorientable surfaces]{Alternating links on nonorientable surfaces and {K}lein-bottly alternating links}

\author{Jessica S.~Purcell}
\address{School of Mathematics, Monash University, Clayton, VIC 3800, Australia}
\email{jessica.purcell@monash.edu}

\author{Lecheng Su}
\address{School of Mathematics, Monash University, Clayton, VIC 3800, Australia}
\email{lecheng.su@monash.edu}



\begin{abstract}
It has been known for several decades that classical alternating links in the 3-sphere have nice hyperbolic geometric properties. Recent work generalises such results to give hyperbolic geometry of links with alternating projections onto any surface in very general 3-manifolds. However, the most general results require an orientable projection surface. In this paper, we extend to alternating links on nonorientable projection surfaces. As an application, we study Klein bottly alternating links in prism manifolds, which are a natural generalisation of Adams' toroidally alternating links in lens spaces.
\end{abstract}

\maketitle

\section{Introduction}\label{Sec:Intro}

An alternating link has a diagram for which crossings alternate over and under when following each component of the link. Link complements can be studied using geometric techniques, for example by work of Thurston~\cite{thurston1982three}, and these techniques are particularly effective in the alternating case. In 1984, Menasco~\cite{Menasco:AltLinks} showed that any classical link in the 3-sphere with a prime alternating diagram is either a $(2,q)$-torus knot, or it is hyperbolic. Similar hyperbolicity results have been obtained by Adams~\cite{Adams:Toroidally}, Hayashi~\cite{Hayashi}, Ozawa~\cite{Ozawa:GenAlt}, Howie~\cite{howie2017characterisation}, and Adams and Chen~\cite{AdamsChen} for alternating links on orientable surfaces in other settings. Howie and Purcell have recent results that are quite broad: they study geometric properties of links with alternating projections onto orientable surfaces in very general 3-manifolds~\cite{HowiePurcell}. 

However, much less is known for the geometry of alternating links on nonorientable surfaces. In ~\cite{adams2019hyperbolicity, adams2023augmented}, Adams \emph{et al} study alternating links in thickened, possibly nonorientable surfaces. The ambient 3-manifolds of Howie and  Purcell are more general, but only apply to orientable projection surfaces. In this paper, we extend the results of Howie and Purcell to nonorientable projection surfaces. This reproduces many results of~\cite{adams2019hyperbolicity} and~\cite{adams2023augmented}, but also extends the work beyond links in thickened surfaces to links in more general 3-manifolds. One result is the following:

\begin{theorem}\label{Thm:hyp}
Let $\pi(L)$ be a weakly generalised alternating diagram of a link $L$ on a nonorientable projection surface $S$ in a compact, orientable, irreducible 3-manifold $M$.
Suppose that the representativity satisfies $r(\pi(L),S)>4$, $\pi(L)$ is cellular, and $M\cut S$ is atoroidal and contains no essential annulus $A$ with $\partial A \subset \partial M$. If $S$ is not a projective plane $P^2$, then $M\setminus L$ is hyperbolic. If $S\cong P^2$, then $M\setminus L$ is hyperbolic if and only if $\pi(L)$ is not a string of bigons arranged end to end. 
\end{theorem}

Careful definitions of terminology can be found in \Cref{Sec:Def}, but briefly, the conditions are mild, automatically satisfied if $\pi(L)$ is cellular in $M$ a twisted $I$-bundle as in~\cite{adams2019hyperbolicity}. Our method of proof for Theorem~\ref{Thm:hyp} is to show $M\setminus L$ is doubly covered by a weakly generalised alternating link on an orientable surface, in the sense of Howie and Purcell~\cite{HowiePurcell}. This leads to other immediate geometric consequences, for example on volume and checkerboard surfaces; see \Cref{Sec:Def}.

Some work has been done in classifying knot diagrams on nonorientable surfaces. Drobotukhina~\cite{Drobotukhina:KnotsProjPlane} tabulated knots projecting onto the projective plane with up to six crossings. Matveev and Nabeeva~\cite{MatveevNabeeva} have tabulated knots projecting on the Klein bottle with up to three crossings. Bourgoin introduced a family called twisted links~\cite{Bourgoin}, which lie in thickened, possibly nonorientable surfaces, generalising the virtual links introduced by Kauffman~\cite{Kauffman:Virtual}.

In a special case we take $S$ to be the Klein bottle and $M$ to be the class of 3-manifolds with finite fundamental group that contain an embedded Klein bottle. These are prism manifolds, investigated by Rubinstein~\cite{rubinstein19793}; see also Scott~\cite{scott1983geometries}. We define a class of \emph{Klein-bottly alternating links}, discussed in \Cref{Sec:Klein}, generalising Adams' toroidally alternating links~\cite{Adams:Toroidally}. We prove a meridian lemma, \Cref{Prop:MerLem}, analogous to the meridian lemma in the toroidally alternating case. This allows us to extend \Cref{Thm:hyp} to more Klein-bottly alternating links, with no requirements on representativity, as follows.

\begin{named}{\Cref{Thm:Kleinbottly}}
\Copy{Statement:Kleinbottly}{
Let $L$ be a geometrically prime, Klein-bottly alternating knot in the prism manifold $M(p,q)$. Then $M(p,q)\setminus L$ is hyperbolic.}
\end{named}

Our methods generalise those of Adams~\cite{Adams:Toroidally} via an extension of tools of Howie and Purcell to nonorientable surfaces~\cite{HowiePurcell}; see also Purcell and Tsvietkova~\cite{PurcellTsvietkova:Standard}. The idea is to decompose the link complement into simpler pieces. In the classical setting, this is done by cutting along checkerboard surfaces, as pioneered by Thurston for a few links~\cite{Thurston:Notes}, and studied carefully by Menasco~\cite{menasco1983polyhedral}. We cut along nonorientable projection surfaces, and use normal surface theory to study geometry.

\subsection{Acknowledgments}
We thank Colin Adams, David Futer, Josh Howie and Stephan Tillmann for helpful discussions. Purcell was supported in part by the Australian Research Council, grant DP240102350.

\section{Links on nonorientable surfaces}\label{Sec:Def}
Throughout, consider piecewise linear maps, so knots and surfaces are tame. 

\subsection{Nonorientable surfaces}

We recall a few facts on nonorientable surfaces and their thickenings in orientable 3-manifolds, partly to define notation. For further information, see for example Martelli~\cite[Section~11.4]{Martelli}.

Let $S$ be a closed, connected nonorientable surface.
Let $\widetilde{S}$ denote the \emph{orientable double cover} of $S$. It is an orientable surface with 2-fold covering map $p\from \widetilde{S}\to S$.
Let $\tau\from \widetilde{S} \to \widetilde{S}$ denote the deck transformation. Thus $\tau$ is a homeomorphism of $\widetilde{S}$ and $p\circ\tau(x) = p(x)$ for all $x\in\widetilde{S}$.

Let $I=[-1,1]$ be the closed interval, and $\iota\from I\to I$ the automorphism $\iota(t)=-t$. The \emph{twisted $I$-bundle} over $S$ is defined to be $S \widetilde{\times} I= (\widetilde{S} \times I) / (\tau,\iota)$. Let $q\from \widetilde{S} \times I \to S \widetilde{\times} I $ denote the quotient map. This is also a double cover. Denote points in $S\widetilde{\times}I$ by $[(x,t)]$, where $[\cdot]$ denotes equivalence class.

Suppose $S$ is embedded in an orientable 3-manifold $M$. Then a regular neighborhood of $S$ in $M$ is homeomorphic to the twisted $I$-bundle $S \widetilde{\times} I$. See, for example \cite[Proposition~1.1.12]{Martelli}. 
There is an embedding of $S$ into $S\widetilde{\times} I$ at level $t=0$: it takes $x$ in $S$ to $[(x,0)]$. Denote the image of the embedding by $S_0$ in $S\widetilde{\times} I$. 
Similarly for $t\neq 0$ in $[-1,1]$, there is an embedding of $\widetilde{S}$ taking $y$ in $\widetilde{S}$ to $[(y,t)]$. Denote the image of this embedding by $\widetilde{S}_t$; note $\widetilde{S}_t =\widetilde{S}_{-t}$. Finally, observe that the boundary of $S\widetilde\times I$ is $\widetilde{S}_{\pm 1}$, homeomorphic to $\widetilde{S}$.

\begin{definition}\label{Def:GenProjSfce}
A \emph{connected nonorientable projection surface} consists of a closed connected nonorientable surface $S$, a compact orientable irreducible 3-manifold $M$, and a piecewise linear embedding of $S$ into $M$ such that $M\setminus S$ is irreducible. 
\end{definition}

Definition~\ref{Def:GenProjSfce} extends \cite[Definition~2.1]{HowiePurcell}, except we restrict to connected surfaces for ease of exposition. 

A \emph{knot} in $S \widetilde{\times} I$ is a circle embedded in $S \widetilde{\times} I$. A \emph{link} $L$ is a disjoint union of circles embedded in $S \widetilde{\times} I$. We say two links are \emph{equivalent} if they are ambient isotopic. If $L$ is a link in the double cover $S\widetilde{\times} I$, then $\widetilde{L}\coloneq q^{-1}(L)$ is a link in $\widetilde{S}\times I$.

\begin{definition}[Generalised diagram]
A link $L \subset S \widetilde{\times} I \subset M$ can be projected onto $S_0$ by $\pi \from S \widetilde{\times} I \to S_0$, where $\pi([(x,t)])=[(x,0)]$. We assume transversality, so $\pi(L)$ is a 4-valent graph on $S_0$. We say the image $\pi(L)$ is a \emph{generalised diagram}. 
\end{definition}

Let $\widetilde{\pi}\from \widetilde{S}\times I \to \widetilde{S}$ be the projection taking $(x,t)$ to $x$. Then $\widetilde{\pi}(\widetilde{L})$ is a generalised diagram of $\widetilde{L}$ on $\widetilde{S}$, in the sense of~\cite{HowiePurcell}.
Note $\widetilde{\pi}(\widetilde{L})$ is a 4-valent graph on $\widetilde{S}$, by transversality. The double cover $p\from \widetilde{S}\to S$ takes $\Ddiagram$ to $\pi(L)$. Similarly $p^{-1}(\pi(L))=\widetilde{\pi}(\widetilde{L})$. For every crossing $X$ on the diagram $\pi(L)$, there are two crossings on $\widetilde{\pi}(\widetilde{L})$ projecting to $X$. 

\begin{definition}[Weakly prime]\label{Def:WeaklyPrime}
A generalised link diagram $\pi(L)$ on a nonorientable projection surface $S$ is \emph{weakly prime} if whenever $D\subset S$ is a disk with $\bdy D$ intersecting $\pi(L)$ transversely exactly twice, $\pi(L)\cap D$ is a single embedded arc.
\end{definition}

\begin{example}[Knots and links on a Klein bottle]
The simplest closed nonorientable surface is the Klein bottle, which we represent as a square with side-pairings as in \Cref{Fig:KleinLink}, left. That figure shows the diagram of a 3-component link in the twisted $I$-bundle over a Klein bottle.

\begin{figure}
  \includegraphics{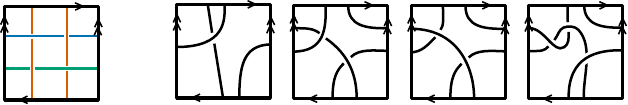}
  \caption{Left to right: A three component link on the Klein bottle, and four knots on the Klein bottle from the Matveev--Nabeeva table~\cite{MatveevNabeeva}; left to right shows $1_1$, $2_6$, $2_5$ and $3_{16}$}
  \label{Fig:KleinLink}
\end{figure}

Diagrams of knots with up to 3 crossings in the twisted $I$-bundle over a Klein bottle were tabulated by Matveev and Nabeeva~\cite{MatveevNabeeva}.  
They restricted to \emph{cellular} diagrams, i.e.\ in which complementary regions are contractible (they called this~\emph{not reducible}), and weakly prime diagrams as in Definition~\ref{Def:WeaklyPrime} (they called~\emph{not composite}). Under these restrictions, they found exactly 17 possible diagrams (graphs) with up to three crossings (4-valent vertices), showed this leads to at most 33 knots, and proved that at least 29 and at most 31 of the knots are distinct.
\Cref{Fig:KleinLink} (right) shows four of their knots. \Cref{Fig:DoubleKleinLink} shows double covers. 
\end{example}

\begin{figure}
  \includegraphics{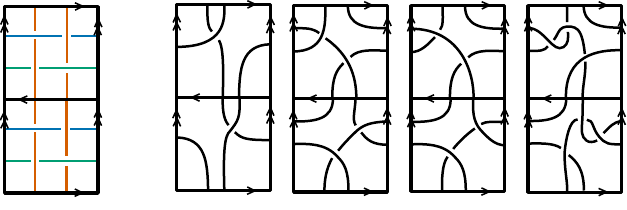}
  \caption{Diagrams $\widetilde{\pi}(\widetilde{L})$ on the orientable double cover of the Klein bottle, for the link and knots of Figure~\ref{Fig:KleinLink}.}
  \label{Fig:DoubleKleinLink}
\end{figure}

\subsection{Alternating knots and links}

In the study of classical alternating knots, or knots on orientable surfaces, there are globally well-defined notions of `over-crossing' and `under-crossing'. That is not the case on nonorientable surfaces, so the definition of alternating links on nonorientable surfaces is not completely trivial. There are a few ways to overcome this. For example, Matveev and Nabeeva view the Klein bottle as a square with sides identified~\cite{MatveevNabeeva}. Since a square is orientable, they define their crossing information on the square. We define alternating via a definition of over--under from the point of view of local discs and oriented curves. In this subsection, we carefully state three definitions of alternating knots on nonorientable surfaces, and show that they are equivalent.

First we define over- and under-crossings locally. To do so, isotope $L$ to lie on $S_0$ except in a ball neighbourhood of each vertex of $\pi(L)$, which corresponds to a crossing. For any such vertex $x$, let $B_x$ denote a small ball neighbourhood of $x$ in $S_0\subset S\widetilde\times I$, so that $\bdy B_x\setminus S_0$ consists of two hemispheres. We may isotope the two arcs of $L$ meeting at $x$ to run through distinct hemispheres of $B_x \setminus S_0$.
Now consider the double cover of this picture. We may take $B_x$ such that the disc $D_x = B_x\cap S_0$ is evenly covered by two discs in $\widetilde{S}$. Then the preimage under $\pi \circ q$ is the disjoint union of two pillars:
\[(\pi \circ q)^{-1}(D_x) \cong (D_x\times I) \sqcup (\tau(D_x) \times \iota(I)).\]
Note $D_x$ and $\tau(D_x)$ have opposite orientations in $\widetilde{S}$, and $\iota$ also switches orientation. 

\begin{definition} \label{Def:Crossing}
Consider the preimage of one of the strands of $L$ running across $x$ under the covering map $q$. If the preimage consists of two arcs of $\widetilde{L}$, one through $D_x\times [0,1)$ and one through $\tau(D_x) \times (-1,0]$, we call that strand an \emph{over-crossing}. If instead the arcs run through $D_x\times(-1,0]$ and $\tau(D_x)\times[0,1)$, it is an \emph{under-crossing}.
\end{definition}

Observe that over- and under-crossings require a choice of $D_x$ versus $\tau(D_x)$. This is a choice of \emph{local orientation}. The choice can be made continuous over a path. 
If $\gamma$ is 1-sided, i.e.\ with regular neighbourhood a M\"obius band, local orientations on $\gamma(0)$ and $\gamma(1)$ will be opposite; traversing $\gamma$ twice returns to the same local orientation. If the curve is 2-sided, $\gamma(0)$ and $\gamma(1)$ have the same local orientation. 

\begin{definition}[Strand alternating]\label{Def:Strand}
Choose a starting point $x$ on $\pi(L)$, and a local orientation on $D_x$. The component of $L$ on which $x$ lies determines a consistent path of orientations for $\pi(L)$, where we take the path to run twice around the component of $\pi(L)$ if the component is 1-sided. Each crossing encountered then inherits a well-defined orientation, hence is an over- or under-crossing, according to Definition~\ref{Def:Crossing}. A generalised diagram $\pi(L)$ is \emph{strand alternating} if, for each component of the link diagram, the crossings encountered always alternate from over-crossing to under-crossing. See \Cref{Fig:AlternatingStrandRegion}, left.
\end{definition}

\begin{figure}
  \includegraphics{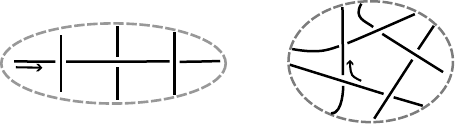}
  \caption{Local pictures of strand and region alternating diagrams}
  \label{Fig:AlternatingStrandRegion}
\end{figure}

\begin{definition}[Region alternating] 
Consider a region $R$ of $S\setminus\pi(L)$. Its boundary forms a curve $\gamma_R\from [0,1]\to S$, and thus defines an orientation.
A generalized diagram $\pi(L)$ is \emph{region alternating} if, for each component $R$ of $S\setminus\pi(L)$, the crossings we encounter when traveling along
$\gamma_R$ always run either from over to under, or under to over. See Figure~\ref{Fig:AlternatingStrandRegion}, right.
\end{definition}

\begin{lemma}\label{Lem:StrandRegion}
  Strand alternating and region alternating coincide. 
\end{lemma}
\begin{proof}
  Each arc of the diagram between vertices of $\pi(L)$ runs from over-crossing to under-crossing, or vice versa, if and only if the same is true for the boundary of adjacent regions at that arc. Since both strand and region alternating are defined by how such arcs are assembled, the diagram will be strand alternating if and only if it is region alternating. 
\end{proof}

\begin{definition}[S-alternating]
A generalized diagram $\pi(L)$ is \emph{$S$-alternating} if the diagram $\widetilde{\pi}(\widetilde{L})$ is alternating on $\widetilde{S}$.
\end{definition}

\begin{theorem}\label{Thm:AlternatingDefs}
For a generalised diagram on a nonorientable surface, strand alternating, region alternating, and $S$-alternating are all equivalent.
\end{theorem}

\begin{proof}
The diagram $\pi(L)$ is strand alternating if and only if each component of $L$ alternates running from over- to under-crossing in an associated orientation as in Definition~\ref{Def:Strand}. 
The path with its orientation lifts to a consistently oriented path on $\widetilde{S}\times I$, with the orientation of Definition~\ref{Def:Strand} agreeing with the orientation on $\widetilde{S}$. 
Thus in $\widetilde{S}$, $\widetilde{L}$ alternates running from over- to under-crossings if and only if the same is true for $\pi(L)$. Hence strand alternating is equivalent to $S$-alternating. 

Region alternating follows similarly, or by \reflem{StrandRegion}. 
\end{proof}

We say that a link in a nonorientable surface has an \emph{alternating diagram} if its diagram is either strand alternating, region alternating, or $S$-alternating. The link is \emph{alternating} if it has an alternating diagram.

\begin{example}
Of Matveev and Nabeeva's knots on the Klein bottle with up to three crossings~\cite{MatveevNabeeva}, only the knots $2_6$ and $3_{16}$ are alternating; see Figure~\ref{Fig:KleinLink} (right) and their double covers in \reffig{DoubleKleinLink}.
\end{example}

\begin{lemma}\label{Lem:DoubleWeaklyPrime}
Let $\pi(L)$ be an alternating generalised diagram on a nonorientable projection surface $S$, embedded in a compact, orientable, irreducible 3-manifold $M$. Then $\pi(L)$ is weakly prime if and only if $\Ddiagram$ is weakly prime on $\widetilde{S}$.
\end{lemma}

\begin{proof}
Suppose $D$ is a disc on $\widetilde{S}$ with $\bdy D$ meeting the diagram $\Ddiagram$ exactly twice. If $\tau(D)\cap D=\emptyset$, then $D$ projects under the double cover $p \from \widetilde{S}\to S$ to a disc on $S$ meeting the diagram $\pi(L)$ exactly twice, and $p(D)$ has no crossings in its interior if and only if $D$ also has no crossings in its interior.

If $\tau(D) \cap D \neq \emptyset$, then $D\setminus\tau(D)$ consists of one or more disc components. For each component $E$ of $D\setminus\tau(D)$, the interior of $E$ is disjoint from $\tau(E)$, because $\tau$ is a local homeomorphism and $E$ is an embedded disc. We may isotope $\bdy E$ slightly so that $E$ is disjoint from $\tau(E)$.

We claim that some component of $D\setminus \tau(D)$ is a disc $E$ with boundary meeting the diagram $\Ddiagram$ exactly twice. The boundary of $D\setminus \tau(D)$ consists of arcs of $\bdy D\setminus \tau(D)$ and arcs $\bdy\tau(D)\cap D$. Note that if $\Ddiagram$ meets $\bdy D\setminus \tau(D)$, then it meets a disc component $E$ of $D\setminus \tau (D)$. Then $\Ddiagram$ must enter and exit $E$, hence it meets $\bdy E$ at least twice. Suppose $\Ddiagram$ meets $\bdy D$ in $\bdy D \cap \tau(D)$. Then applying $\tau$, $\Ddiagram$ meets $\tau(\bdy D)$ in $\tau (\bdy D\cap \tau(D)) = \tau(\bdy D)\cap D$. Because $\tau$ is a homeomorphism, this is $\bdy(\tau (D))\cap D$, which lies on the boundary
of $D\setminus \tau(D)$. Thus again $\Ddiagram$ must enter and exit a disc component $E$ of $D\setminus \tau(D)$, so it meets $\bdy E$ at least twice. Because $\bdy D$ meets $\Ddiagram$ only twice, if $\bdy E$ meets $\Ddiagram$ more than twice, the intersections must lie on $\bdy(\tau (D))\cap D$. We claim that $\partial(\tau(D))\cap D$ cannot meet $\Ddiagram$ more than twice. This is because each such intersection lies on $\partial \tau(D)$, hence corresponds to an intersection with $\tau(\bdy \tau(D))=\bdy D$, and we know $\Ddiagram$ meets $\bdy D$ exactly twice. Thus there is a disc component $E$ of $D\setminus\tau(D)$ meeting the diagram exactly twice.
After a slight isotopy, its projection to $S$ is embedded and meets the diagram twice. It bounds crossings in its interior if and only if the same is true for $E$.
\end{proof}

\begin{definition}[Reduced alternating]
A generalised link diagram $\pi(L)$ on a nonorientable projection surface $S$ is \emph{reduced alternating} if $\pi(L)$ is weakly prime, alternating, and each component projects to at least one crossing on $S$.
\end{definition}

\begin{definition}[Representativity]
Let $S$ be a nonorientable projection surface in a compact, orientable irreducible 3-manifold $M$, and let $L$ be a link in $S\widetilde{\times} I$ in $M$. 
Define the \emph{representativity} $r(\pi(L), S)$ to be the minimum number of intersections between $\widetilde{\pi}(\widetilde{L})$ and any circle that bounds a compression disc for $\widetilde{S}=\partial(S\widetilde{\times}I)$ in $M\cut S:=\overline{M\setminus N(S)}$. If there are no compression discs for $\widetilde{S}$, then $r(\pi(L), S)=\infty$.
\end{definition}

\begin{example}
The representativity depends on the embedding of $S$ and the manifold $M$. Consider again the knots and links in Figure~\ref{Fig:KleinLink}, on the Klein bottle $S$. If $M = S\widetilde{\times}I$, then there are no compressing discs for $S$ in $M$, and the representativity satisfies $r(\pi(L),S)=\infty$ for every diagram. However, we will see in \refsec{Klein} that we can also embed $S$ in a prism manifold $M(p,q)$. Removing $S\widetilde{\times} I$ from $M(p,q)$ gives a solid torus, which has a compression disc, and hence we obtain finite representativity. The representativity will depend on $M(p,q)$ and the diagram.
\end{example}

On an orientable surface, an alternating link diagram is \emph{checkerboard colourable} if each complementary region of the diagram graph can be oriented such that the induced orientation on each boundary component runs from over to under or vice-versa. If each complementary region is a disc, that is, the diagram is \emph{cellular}, the diagram must be checkerboard colourable. 
	
A 4-valent graph $\Gamma$ on a nonorientable surface $S$ lifts to a 4-valent graph $\widetilde{\Gamma}$ on the orientable double cover $\widetilde{S}$. The graph $\Gamma$ is \emph{checkerboard colourable} on $S$ if its lift $\widetilde{\Gamma}$ admits a checkerboard coloring on $\widetilde{S}$, and the coloring is preserved by the deck transformation $\tau$.

\begin{example}
Of Matveev and Nabeeva's 17 diagrams of knots on the Klein bottle~\cite{MatveevNabeeva}, only two are checkerboard colourable, and these give only three knots that are checkerboard colourable (two with the same diagram, one alternating, and one not). These are the knots $2_6$, $2_5$, and $3_{16}$ shown in \reffig{KleinLink}. However, because they are cellular, all these diagrams are checkerboard colourable on $\widetilde{S}$. 
\end{example}

\begin{definition}
A link on a nonorientable surface is \emph{weakly generalised alternating} if it is reduced alternating and has representativity $r(\pi(L),S)\geq 4$, and $\Ddiagram$ is checkerboard colourable on $\widetilde{S}$. 
\end{definition}
Note this agrees with~\cite[Definition~2.9]{HowiePurcell} in the orientable case. Furthermore, we only require $\Ddiagram$ to be checkerboard colourable on $\widetilde{S}$, not $\pi(L)$ on $S$. 

Let $S$ be a connected nonorientable projection surface embedded in a compact orientable irreducible 3-manifold $M$. Define $\widetilde{M}$ to be the double of $M\cut S$ along $\widetilde{S}$. That is, $\widetilde{M}$ is obtained by gluing two copies of $M\cut S$ by reflection in $\bdy(M\cut S)=\widetilde{S}$. 

\begin{proposition} \label{Prop:doublecover}
Let $S$ be a connected nonorientable projection surface embedded in a compact orientable irreducible 3-manifold $M$. Let $L$ be a weakly generalised alternating link with a generalised diagram on $S$. Then:
\begin{enumerate}
\item $\widetilde{M}$ is a double cover of $M$. 
\item $\widetilde{M} \setminus \widetilde{L}$ is a double cover of $M\setminus L$.
\end{enumerate}	
\end{proposition}

\begin{proof}
Take $N$ to be a tubular neighborhood of $S$ in $M$, and $\widetilde{N}$ a tubular neighborhood of $\widetilde{S}$ in $\widetilde{M}$. Since $M$ is orientable, $N$ is homeomorphic to $S\widetilde{\times}I$. Furthermore, $\widetilde{N}$ is homeomorphic to $\widetilde{S}\times I$. We can construct a map $q\from \widetilde{N} \to N$ that is both the quotient map for the twisted $I$-bundle $N$ and a double cover. Since $\widetilde{M}\setminus\widetilde{N}\cong \widetilde{M}\cut \widetilde{S}$, the manifold $M\setminus \widetilde{N}$ is homeomorphic to two disjoint copies of $M\cut S$ by definition of doubling. Note $M\setminus N\cong M\cut S$. Then the map $\widetilde{M}\setminus \widetilde{N} \to M\setminus N$, mapping each copy of $M\cut S$ by the identity, is also a double cover. It agrees with $q$ on the common boundary $\widetilde{S}\times \{\pm 1\}$, giving a continuous map that is a double cover of $M$.
	
For the link complement $M\setminus L$ and the doubled complement $\widetilde{M} \setminus \widetilde{L}$, we can also cut the complement along a neighborhood of the projection surface. We write $\widetilde{M}\setminus \widetilde{L}=(\widetilde{M}\cut \widetilde{S})\cup (\widetilde{S} \times I\setminus \widetilde{L})$, where  $\widetilde{M}\cut \widetilde{S}$ is two copies of $M\cut S$. Furthermore, $M\setminus L=(M\cut S)\cup(S\widetilde{\times} I \setminus L)$. Thus we only need to show that $\widetilde{S}\times I\setminus \widetilde{L}\to S\widetilde{\times}I\setminus L$ is a double cover. This is the restriction of the double covering and quotient map $q: \widetilde{S}\times I\to S\widetilde{\times}I$. Since $\widetilde{L}=q^{-1}(L)$ by definition, the restriction of $q$ on $\widetilde{S}\times I\setminus \widetilde{L}$ is indeed a double cover.
\end{proof}

\begin{theorem}\label{Thm:WGA}
If $\pi(L)$ is a weakly generalised alternating diagram on a nonorientable surface $S$ embedded in a compact orientable irreducible 3-manifold $M$, then $\widetilde{L}$ is a weakly generalised alternating link in $\widetilde{M}$.
\end{theorem}

\begin{proof}
We check each requirement for a weakly generalised alternating link on an orientable surface, in~\cite[Definition~2.9]{HowiePurcell}. 
Consider $\Ddiagram$ on $\widetilde{S}$. The diagram $\Ddiagram$ is alternating on $\widetilde{S}$ by Theorem~\ref{Thm:AlternatingDefs}, and weakly prime by  Lemma~\ref{Lem:DoubleWeaklyPrime}. It meets the projection surface $\widetilde{S}$ in at least one crossing by our definition of reduced alternating and the fact that we only consider connected projection surfaces. The diagram $\Ddiagram$ is checkerboard colourable by definition. Its representativity is at least four by definition. Therefore, the diagram $\Ddiagram$ is a weakly generalised alternating diagram on $\widetilde{S}$ in $\widetilde{M}$.
\end{proof}

\begin{proof}[Proof of Theorem~\ref{Thm:hyp}]
By Theorem~\ref{Thm:WGA}, $\widetilde{L}$ is weakly generalised alternating. Furthermore, $\widetilde{M}\cut \widetilde{S}$ is atoroidal and anannular since $M\cut S$ is atoroidal and anannular. If $S\ncong P^2$, the surface $\widetilde{S}$ has genus at least one, implying $\widetilde{M}\setminus \widetilde{L}$ is hyperbolic by~\cite[Theorem~4.2]{HowiePurcell}. If $S\cong P^2$, then by irreducibility of $M$, $\Ddiagram$ is a classical alternating link in $S^3$, and by~\cite{Menasco:AltLinks} it is hyperbolic if and only if it is not a $(2,q)$-torus link, if and only if $\pi(L)$ is not a string of bigons arranged end to end. 
Proposition~\ref{Prop:doublecover} implies $M\setminus L$ is double-covered by $\widetilde{M}\setminus \widetilde{L}$, hence it is also hyperbolic if and only if $\widetilde{M}\setminus \widetilde{L}$ is hyperbolic.
\end{proof}

We obtain other immediate geometric consequences of \refthm{WGA}, for example for hyperbolic volumes and the geometry of embedded surfaces, which we record below. We omit definitions since we do not need them here; they are in ~\cite{HowiePurcell}. 

\begin{corollary}
Let $L$, $S$, $M$ be as in Theorem~\ref{Thm:hyp}, with $S\ncong \RR P^2$. Suppose further that if $\bdy M$ is nonempty, then $\bdy M$ is incompressible in $M$. Then:
\begin{enumerate}
\item[(1)] If $M\setminus L$ admits two checkerboard surfaces, then they are quasifuchsian. Otherwise, the checkerboard surfaces of $\Ddiagram$ project to an immersed quasifuchsian surface, with self intersections at crossing arcs. 
\item[(2)] The volume of $M\setminus L$ is bounded below:
 \[ vol(M\setminus L)\geq \frac{v_8}{4}(tw(\Ddiagram)-2 \chi(S)-2 \chi(\bdy M))\]
\end{enumerate}
\end{corollary}

This is a direct consequence of Theorem~\ref{Thm:WGA} and ~\cite[Theorem~6.10 and~Corollary~5.9]{HowiePurcell}. Item (1) extends similar results in the classical case~\cite{Adams:Quasifuchsian, FKP:quasifuchsian}. Item (2) extends work of 
Lackenby on volumes~\cite{Lackenby:AltVol}.

\section{Chunk Decomposition}\label{Sec:Chunk}
We can decompose any alternating link complement on a nonorientable projection surface into chunks, defined by Howie and Purcell~\cite{HowiePurcell}. These are a generalisation of the checkerboard decomposition of Menasco~\cite{menasco1983polyhedral}; see also Aitchison and Rubinstein~\cite{AitchisonRubinstein}, Lackenby~\cite{Lackenby:AltVol}, and Futer and Gu{\'e}ritaud~\cite{FuterGueritaud} for a related decomposition. The decomposition simplifies the process of identifying embedded surfaces in the link complement, and we will use it in the next section. 

A \emph{chunk} $C$ is a compact, orientable, irreducible 3-manifold with boundary $\partial C$ containing an embedded (possibly disconnected) non-empty graph $\Gamma$ with all vertices having valence at least 3. We allow components of $\partial C$ to be disjoint from $\Gamma$, provided any such component is incompressible. 
Regions of $\partial C \setminus \Gamma$ are called \emph{faces}. A component of $\partial C$ disjoint from $\Gamma$ is called an \emph{exterior face}. Other faces are called \emph{interior faces}.
A \emph{truncated chunk} is a chunk for which a regular neighborhood of each vertex of $\Gamma$ has been removed. The newly produced faces and edges after removal are called \emph{truncation faces} and \emph{truncation edges}, as in Purcell and Tsvietkova~\cite{PurcellTsvietkova:Standard, PurcellTsvietkova:Polynomial}. 
	
A \emph{(truncated) chunk decomposition} of a 3-manifold $M$ is a decomposition of $M$ into (truncated) chunks (possibly one), such that $M$ is obtained by gluing chunks by homeomorphisms of faces on $\partial C$, with edges mapping to edges homeomorphically.

\begin{construction}\label{Const:ChunkDecomp}
Let $L$ be a link embedded in a compact, irreducible, orientable 3-manifold $M$, with a generalised diagram $\pi(L)$ that is reduced alternating on a nonorientable projection surface $S$. Suppose $\pi(L)$ has $k$ crossings.
Consider $M\cut S$. Observe that  $\bdy(M\cut S)\setminus \bdy M$ is homeomorphic to $\widetilde{S}$. The chunk $C$ will be homeomorphic to $M\cut S$; we need to determine the graph $\Gamma$ on its boundary $\widetilde{S}$. 
		
In the classical alternating setting, informally, one finds the faces of the decomposition by viewing the two ball components of $S^3$ cut along the plane of projection as two balloons, and one expands those balloons until they bump in regions complementary to the link diagram. In the setting of $\pi(L)$ on nonorientable $S$, we expand a balloon shaped like $M\cut S$. That is, we expand $M\cut S$ towards the removed twisted $I$-bundle, and it bumps itself in regions complementary to the link diagram. These will be faces, glued where regions bump. See \Cref{Fig:Bumping}, left.

Continue to expand $M\cut S$ towards $S$. The faces will meet in pairs at crossing arcs. For a fixed crossing, and a fixed local orientation at that crossing, two pairs of faces will meet at an over-crossing, and the opposite pairs of faces will meet at the under-crossing. The crossing arcs will be ideal edges of $\Gamma$. Put two copies of each crossing arc onto either side of the crossing; that is, put two arcs at one lift to $\widetilde{S}$ and others at the other lift. See \Cref{Fig:Bumping}. This gives a total of $4k$ ideal edges, which will glue into $k$ crossing arcs.
		
\begin{figure}
  \includegraphics{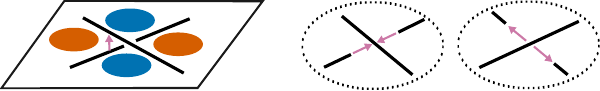}
  \caption{Left: Local picture of the bumping pattern near a crossing. Right: Locations of edges, shown with orientation on a disc $D$ and $\tau(D)$. Note the four pink edges are identified to a single crossing arc.}
  \label{Fig:Bumping}
\end{figure}
		
The boundary of $M\cut S$ is now decorated with faces and edges, as well as remnants of the link diagram. Each remnant of link diagram has an endpoint at an under-crossing in $\widetilde{S}$,  and runs through an over-crossing to end at a second under-crossing. Shrink this remnant to a single ideal vertex, located at the position of the over-crossing. This will pull the ideal edges at the under-crossings into the over-crossing, giving a 4-valent vertex. See \Cref{Fig:ShrinkExpand}. 
		
\begin{figure}
  \begin{overpic}{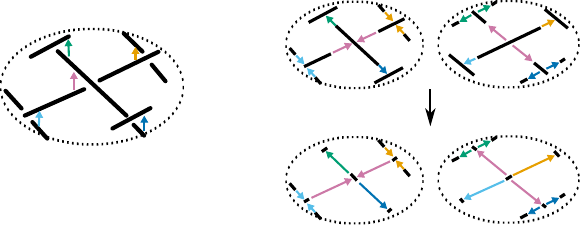}
    \put(50,15){$D$}
    \put(95,15){$\tau(D)$}
  \end{overpic}
  \caption{Left: Portion of link diagram showing adjacent crossings, with crossing arcs indicated. Right top: corresponding ideal edges; Right bottom: shrink remnants of link to ideal vertices.}
  \label{Fig:ShrinkExpand}
\end{figure}
The chunk decomposition is now complete. We summarize its properties.
\end{construction}

\begin{theorem}[Chunk decomposition]\label{Thm:ChunkDecomp}
Let $L$ be an alternating link in a compact, irreducible, orientable 3-manifold $M$, with alternating generalised diagram $\pi(L)$ on a nonorientable projection surface $S$. Then $M\setminus L$ admits a chunk decomposition.
\begin{enumerate}
\item[(1)] The single chunk is homeomorphic to $M\cut S$. 
\item[(2)] The embedded graph on $\bdy(M\cut S) = \widetilde{S}$ is identical to the diagram graph $\widetilde{\pi}(\widetilde{L})$, with ideal vertices corresponding to the crossings.
\item[(3)] Each region of $\widetilde{S}\setminus\widetilde{\pi}(\widetilde{L})$ is glued to its translate under the covering map $\tau$. The gluing is by first applying $\tau$, then rotating either clockwise or anticlockwise to the next adjacent edge, with direction depending on whether the boundary of the region runs from over- to under-crossing in a clockwise or anticlockwise direction. 
\item[(4)] Edges correspond to crossing arcs, and are glued in fours. At each ideal vertex, two opposite edges are glued together.
\end{enumerate}
\end{theorem}

\begin{proof}
The chunk is homeomorphic to $M\cut S$, which is a compact, orientable, irreducible 3-manifold with boundary $\widetilde{S}$, the orientable double cover of $S$. 
	
Regions of the graph $\Gamma$ on $\widetilde{S}$ come from regions of $\widetilde{S}\setminus\widetilde{\pi}(\widetilde{L})$, which are glued to regions identified across $S$, or via $\tau(S)$.
Note that $\tau$ reverses the orientation of $R$, as necessary for gluing the boundary of $M\cut S$.
Now note that edges on $R$ and $\tau(R)$ are identified by a shift in the clockwise or anticlockwise direction, as shown in \Cref{Fig:ShrinkExpand}. This is due to the fact that locally, on one side of a disc on $S$ the over-strand separates the two edges identified to the crossing arc, but on the other side the under-strand separates them. Thus we must follow the translation by $\tau$ by a rotation, with the rotation in opposite directions in adjacent faces.

The edges of the graph $\Gamma$ separate regions at crossings, and glue in sets of four to crossing arcs as described above. Their endpoints lie on remnants of the link, which have been collapsed to ideal vertices lying at over-crossings. When these remnants shrink to ideal vertices, the endpoint of an ideal edge is pulled along the link diagram, to take the place of the diagram edge in $\widetilde{\pi}(\widetilde{L})$, giving an isomorphism between that graph and $\Gamma$. Since the diagram is alternating, this gives 4-valent vertices, satisfying the definition of a chunk. 
\end{proof}

In the classical case, Thurston likened the face rotation (3) to a gear-rotation.

\begin{example}
The example of the decomposition for the link on the left of Figure~\ref{Fig:KleinLink} is shown in Figure~\ref{Fig:chunk-result}. Note that the faces labeled $1$ in that figure are glued; the one at the top must be translated by $\tau$, which reverses its orientation (flips it horizontally). It then must be rotated in the anticlockwise direction to match edges. 
\begin{figure}
  \centering
  \includegraphics{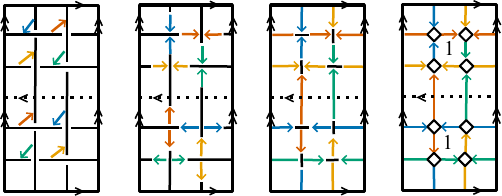}
  \caption{Construction of chunk decomposition for the link on the left of Figure~\ref{Fig:KleinLink}. Left to right: crossing arcs shown, these become ideal edges, shrink remnants of link to ideal vertices, truncate ideal vertices}
  \label{Fig:chunk-result}
\end{figure}
\end{example}

We study how surfaces behave inside each chunk using normal surfaces, generalising normal surfaces from Haken's 1961 work~\cite{haken1961theorie}. Such generalisations are used in many recent results, for example~\cite{FuterGueritaud, HowiePurcell, PurcellTsvietkova:Polynomial, PurcellTsvietkova:Standard}.

Following~\cite[Definition~3.7]{HowiePurcell}, a properly embedded surface $S$ in a truncated chunk $C$ is \emph{normal} if every component of $S\cap C$ is incompressible in $C$, if $S$ and $\bdy S$ are transverse to all faces and edges of $C$, and curves of $S\cap \bdy C$ intersect the faces of $C$ efficiently, namely:
A component of $S\cap \bdy C$ lying entirely in a face does not bound a disc in that face, and arcs of $S \cap\bdy C$ with endpoints on the same edge, or on a truncation edge and an adjacent interior edge, cannot cobound a disc with these edges. A properly embedded surface $(S, \partial S) \subset (M, \partial M)$ is \emph{normal} if, for every chunk $C$ in the chunk decomposition for $M$, $S\cap C$ is normal in $C$.

We will not use all the properties of the definition of normal, but we do use heavily~\cite[Theorem~3.8]{HowiePurcell}, which states that any essential, meridianally incompressible surface embedded in $M$ can be isotoped into normal form.

\section{Klein-bottly alternating links}\label{Sec:Klein}

In a paper from 1994~\cite{Adams:Toroidally}, Adams examines a class of links in lens spaces called \emph{toroidally alternating links}, in which he puts an alternating link diagram on the Heegaard torus of a lens space. In 1978~\cite{rubinstein1978one}, Rubinstein studied one-sided Heegaard splittings of 3-manifolds as a generalisation of Heegaard splittings. That is, 
when $M$ is a closed orientable 3-manifold and $S$ a closed nonorientable surface embedded in $M$, the pair $(M,S)$ is called a \emph{one-sided Heegaard splitting} if $M\cut S$ is a handlebody. 

A natural generalisation of toroidally alternating links is to consider a one-sided Heegaard splitting of $M$ along a Klein bottle $S$ with $M\cut S$ a single solid torus, with an alternating diagram on the Klein bottle. In 1979~\cite{rubinstein19793}, Rubinstein examines a class of 3-manifolds with finite fundamental group that contain an embedded Klein bottle. He shows these are exactly the manifolds that admit a one-sided Heegaard splitting with $M\cut K$ a solid torus. These are the \emph{prism manifolds}.

A prism manifold can be realized as $(p,q)$-Dehn filling the torus boundary of a twisted $I$-bundle over a Klein bottle $K$, and we denote it by $M(p,q)$. In the case where $p=1$, after Dehn filling we obtain $L(4q, 2q-1)$. If $p\neq 1$, then we obtain a spherical 3-manifold whose fundamental group is a central extension of a dihedral group. For more information, see Scott~\cite{scott1983geometries} or Lackenby and Schleimer~\cite{LackenbySchleimer}.

Let $L$ be a link in $M=M(p,q)$ that can be isotoped into a regular neighborhood of the Klein bottle $K$ in $M$, which is a twisted $I$-bundle $K\widetilde{\times} I$. The following definition is a direct generalisation of Adams:

\begin{definition}[Klein-bottly alternating links]\label{Def:KBalt}
Suppose that $L$ admits a generalised diagram $\pi(L)$ on the Klein bottle $K$ in $M(p,q)$. Then $L$ is called \emph{Klein-bottly alternating} if the diagram is alternating on $K$ and every nontrivial curve on $\widetilde{K}$ intersects $\Ddiagram$.
\end{definition}

Note that the second condition forces $\pi(L)$ to be cellular. Observe also that there are no representativity requirements. \Cref{Thm:hyp} immediately implies:

\begin{corollary}
Let $L$ be Klein-bottly alternating on $K$ in $M(p,q)$, and suppose that $r(\pi(L), K)>4$. Then the link $L$ is hyperbolic.
\end{corollary}

\begin{corollary}
Fix a cellular diagram of an alternating link on $K$. Then for all but finitely many $(p, q)$, $L$ is hyperbolic in $M(p,q)$.
\end{corollary}

\begin{proof}
The boundary of the compression disc for $M(p,q)$ intersects the diagram the number of times the $(p,q)$-curve on the torus $\widetilde{K}$ meets $\Ddiagram$. For fixed $\Ddiagram$, this intersection number is less than four only for finitely many $p, q$.
\end{proof}

For example, this corollary holds for diagrams of Matveev and Nabeeva~\cite{MatveevNabeeva}.

In his 1994 work, Adams proved a meridian lemma for toroidally alternating knots in lens spaces. We generalise his result to Klein-bottly alternating links.

\begin{proposition}[Meridian Lemma]\label{Prop:MerLem}
Let $L$ be a Klein-bottly alternating link in the prism manifold $M:=M(p,q)$. If $M\setminus L$ contains a closed, orientable, incompressible, meridianally incompressible surface $F$, then $F$ can be isotoped to meet the solid torus $V$ defining $M(p,q)$ in an even number of meridianal discs for $V$.
\end{proposition}

To prove \Cref{Prop:MerLem}, we will use techniques generalised from~\cite{Adams:Toroidally} and~\cite{HowiePurcell}.
Let $M=M(p,q)$ and let $\pi(L)$ be a Klein-bottly alternating diagram on a Klein bottle $K$. Let $C$ be the chunk of $M\setminus L$ of \Cref{Thm:ChunkDecomp}, so $C$ is a solid torus.

\begin{lemma}[Left-right rule]\label{Lem:LeftRight}
A curve that enters a complementary region of $\Ddiagram$ with an over-crossing on the left must exit with an over-crossing on the right.
\end{lemma}

\begin{proof}
This follows directly from our definition of region alternating, analogous to the \emph{alternating property} of Menasco~\cite{Menasco:AltLinks}. 
\end{proof}

Now let $F$ be a closed, orientable, incompressible, meridianally incompressible surface in $M\setminus L$. Using~\cite[Theorem~3.8]{HowiePurcell}, we will assume $F$ has been isotoped into normal form with respect to the chunk decomposition.

\begin{lemma}\label{Lem:Curvealpha}
  No curve of intersection $F\cap \partial C$ bounds a disc on $\partial C$.
\end{lemma}

\begin{proof}
Suppose that there is an innermost curve $\alpha$ of $F\cap \bdy C$ that bounds a disc on $\partial C$. The curve $\alpha$ must intersect interior edges of the chunk by fact that normal form has no closed curves of intersection with faces.
Because the diagram is cellular, $\alpha$ must enter and exit each region. 
By Lemma~\ref{Lem:LeftRight}, $\alpha$ must intersect edges corresponding to two crossings, once with the crossing on one side of $\alpha$, once on the other side. By \Cref{Thm:ChunkDecomp}~(4), the edge that $\alpha$ intersects is identified with another edge on the opposite side of the vertex corresponding to the crossing. Thus there must exist a corresponding intersection arc on the other side of the crossing. Because $\alpha$ is innermost, one of these arcs must also be a part of $\alpha$. However, after gluing the crossing arcs, there exists a meridianal compression disc on $X$, contradicting the meridianal incompressibility of $F$. 
\end{proof}

\begin{lemma}\label{Lem:MeridDisc}
  Components of $F\cap C$ are meridian discs for the solid torus $C$.
\end{lemma}

\begin{proof}
By \Cref{Lem:Curvealpha}, no component of $F\cap C$ bounds a disc on $\bdy C$. By definition of normal, each component of $F\cap C$ is incompressible, so $F\cap C$ consists of annuli and meridianal discs; see for example~\cite{rubinstein1978one} or~\cite[Theorem~8.2]{PurcellTsvietkova:Standard}.

Suppose that there is an annulus component $E$ of $F\cap C$. It must be parallel to $\bdy C$. Assume $E$ is outermost. The boundary $\bdy E$ consists of two parallel curves $A_1$ and $A_2$ on $\bdy C$, bounding an annuls $A\subset \bdy C$. Note they must intersect edges of $\Gamma=\Ddiagram$ by definition of normal. By \Cref{Lem:LeftRight} and outermost, $A_1$ and $A_2$ must intersect the same class of crossing arcs (interior edges). 

If $A_1$ intersects interior edge $e_1$, as in the proof of \Cref{Lem:Curvealpha} some curve of $F \cap \bdy C$ meets an edge $e_2$ adjacent across an ideal vertex to $e_1$. Because $F$ is meridianally incompressible, this cannot be $A_2$. Therefore, we can assume that $A_1$ and $A_2$ both intersect $e_1$. But in this case, there is a boundary compressing disc $D$ for $E$ with an arc of its boundary on $e_1$. We may isotope $F$ through $D$ to remove both intersections of $A_1$, $A_2$ with $e_1$.
Thus we can assume that $F\cap C$ consists only of meridianal discs of $C$.
\end{proof}

\begin{lemma}\label{Lem:EulerChar}
Suppose $F$ is a closed orientable surface of genus $g$ such that $F\cap C$ consists only of meridian discs of the solid torus $C$. Let $v$ denote the number of intersections $F\cap \Ddiagram$, and let $f$ denote the number of meridian discs $F\cap C$. Then $f=2-2g+v$.
\end{lemma}

\begin{proof}
The intersection of $F$ with $C$ gives $F$ a cell structure. The 0-cells are intersections with $\Ddiagram$, which are edges of $\Gamma$. The 1-cells are intersections with regions of $\bdy C\setminus \Gamma$. The 2-cells are the meridian discs of $F\cap C$. Letting $e$ denote the number of 1-cells, we have $2e=4v$ since $\Ddiagram$ is 4-valent. 
We calculate the Euler characteristic $\chi(F)=2-2g=v-e+f$. Reorganizing, $f=2-2g+v$.
\end{proof}

\begin{proof}[Proof of Meridian Lemma, \Cref{Prop:MerLem}]
By \Cref{Lem:MeridDisc}, components of $F\cap C$ are meridian discs for the solid torus $C$. By \Cref{Lem:EulerChar}, the number of meridian discs is $2-2g+v$. Because $\Ddiagram$ is checkerboard coloured, $v$ must be even. 
\end{proof}

We now consider hyperbolicity of Klein-bottly alternating knots. We say a link is \emph{(geometrically) prime} in a 3-manifold $M$ if its complement in $M$ contains no essential meridianal annulus.
Adams proved that geometrically prime, toroidally alternating knots in many lens spaces are hyperbolic~\cite{Adams:Toroidally}. We now show that geometrically prime Klein-bottly alternating knots are hyperbolic. 

\begin{theorem}\label{Thm:Kleinbottly}
Let $L$ be a prime, Klein-bottly alternating knot in $M(p,q)$. Then $M\setminus L$ is hyperbolic. 
\end{theorem}

Note that while \Cref{Prop:MerLem} applied to links, the proof of \Cref{Thm:Kleinbottly} does require that we restrict to knots. 

\begin{proof}[Proof of \Cref{Thm:Kleinbottly}]
By Thurston's hyperbolization theorem, we need only rule out essential spheres, discs, tori, and annuli~\cite{thurston1982three}.

Suppose there is an essential torus $T$ in $M\setminus L$.  \Cref{Lem:EulerChar} implies that $f=v$, meaning each meridianal disc $F\cap C$ contributes one intersection of $T$ and $\Gamma = \Ddiagram$. Then either $\Ddiagram$ is not cellular or not checkerboard colorable, a contradiction. 

Suppose we have an essential sphere $S$ in $M\setminus L$. Then \Cref{Lem:EulerChar} implies that $f=v+2$, which means there will be at least two meridianal discs that have no intersection with $\Gamma$, a contradiction.

Suppose there is an essential compression disc $D$ for $\partial N(L)$ in $\partial(M\setminus N(L))$. The disc $D$ has boundary some $(u,v)$-curve on $N(L)$, where $u$ is the number of times it wraps around a meridian of $N(L)$, and $v$ the number of times it wraps around the longitude, taken to have minimal intersection with $K$ (blackboard longitude). If $v=0$, surger $N(L)$ along $D$ to obtain a 2-sphere embedded in $M(p,q)$ that is non-separating. This contradicts the irreducibility of $M$. Similarly if $v=1$, then $L$ is unknotted, bounding a disc with boundary disjoint from $K$ (using $L$ is a knot, hence 2-sided). Either the disc can be isotoped into $K$, or a curve of $D\cap K$ forms a $(p,q)$ curve on $\widetilde{K}$ that is disjoint from $L$. Either case contradicts Definition~\ref{Def:KBalt} for Klein-bottly alternating.

Suppose $v\geq 2$. Then $N(L)\cup N(D)$ is a lens space minus a ball $B$, where $N(D)$ denotes regular neighborhood of the disc in $M\setminus L$. Since $L$ is completely disjoint from $B$, we can recover $M(p,q)$ by gluing the ball back to $N(L)\cup N(D)$. The result after gluing is a lens space, and the only possibility for $M$ is that $M=L(4q,2q-1)$, where $|q|\geq 1$. The lens space has classical Heegaard splitting into two solid tori $V_1$ and $V_2$. The knot $L$ must be the core curve $V_1$ because the Heegaard splitting of a lens space is unique. By transversality, we may isotope the core of $V_2$ to meet the klein bottle $K$ in a finite number of points. Thus we may isotope the Heegaard torus $\bdy V_2$ such that $V_2$ intersects $K$ in a finite number of meridians for $V_2$. In fact, Rubinstein shows in~\cite[Lemma~9]{rubinstein1978one} that $\bdy V_2$ may be isotoped so that $V_2\cap K$ is a single meridian disc. Since $L$ intersects the meridianal curve of $V_1$ in a single point, far from $\bdy V_2$, by a further isotopy disjoint from $L$, we may transfer $K\cap V_2$ into a Mobius band $X$ following Rubinstein~\cite[Lemma~10]{rubinstein1978one}. Then the core curve of $X$ is a nontrivial curve on $K$ that does not intersect the diagram $\pi(L)$, contradicting Definition~\ref{Def:KBalt}. Thus there is no essential disc. 

Now suppose we have an essential annulus $A$ in $M\setminus L$. The annulus is not meridianal by our assumption that $L$ is geometrically prime. Note that $M\setminus N(L)$ is compact, orientable, irreducible and atoroidal. Then $M\setminus N(L)$ is Seifert fibered; see for example Hatcher~\cite[Lemma~1.18]{hatcher2007notes}. The annulus $A$ can be isotoped to be vertical if $M\setminus N(L)$ is not an $I$-bundle over the torus or Klein bottle; see Hatcher~\cite[Lemma~1.16]{hatcher2007notes}. In this case, the annulus $A$ appears as a union of fibers in the Seifert fibration, with $\partial A$ two of the fibers that lie on $\partial N(L)$. Say one component of $\partial A$ is a $(u,v)$-curve on $\bdy N(L)$ that is not meridianal. Then there exists a Seifert fibered solid torus $V$ where $\bdy V$ agrees with the Seifert fibration of $M\setminus N(L)$. Therefore, we can insert a solid torus to obtain a Seifert fibering of $M=M(p,q)$. However, by Rubinstein~\cite[Proposition~3]{rubinstein19793}, the fundamental group of $M$ can be presented by
$\pi_1(M)=\langle a,b|abab^{-1}, a^p b^{2q} \rangle$, and the fibers of $M$ are homotopy classes of $b^2$ on $\partial N(K)$. This Seifert fibering forces the diagram to have no crossing at all because we can embed the entire knot on the torus $T=\bdy N(K)$, contradicting the fact that $L$ meets any nontrivial curve on the Klein bottle. 

The manifold $M\setminus N(L)$ cannot be an $I$-bundle over the torus because it has only one boundary component. In the case where $M\setminus N(L)$ is an $I$-bundle over the Klein bottle, there is only one Klein bottle in $M$ up to isotopy, so $L$ does not touch the Klein bottle $K$, which contradicts Definition~\ref{Def:KBalt}. Thus there is no essential annulus in $M\setminus L$.
\end{proof}

\bibliographystyle{amsplain}
\bibliography{refs}
\end{document}